\UseRawInputEncoding\documentclass[12pt, reqno]{amsart}
\usepackage[margin=1in]{geometry}
\usepackage{amssymb,latexsym,amsmath,amscd,amsfonts}
\usepackage{latexsym}
\usepackage[mathscr]{eucal}
\usepackage{bm}
\usepackage{mathptmx}
\usepackage{amssymb}
\usepackage{amsthm}
\usepackage{dcolumn}
\usepackage[all]{xy}
\usepackage{enumitem}
\usepackage[utf8]{inputenc}

\def\textmatrix#1&#2\\#3&#4\\{\bigl({#1 \atop #3}\ {#2 \atop #4}\bigr)}
\def\dispmatrix#1&#2\\#3&#4\\{\left({#1 \atop #3}\ {#2 \atop #4}\right)}
\newcommand{\beg}{\begin{equation}}
	\newcommand{\eeg}{\end{equation}}
\newcommand{\ben}{\begin{eqnarray*}}
	\newcommand{\een}{\end{eqnarray*}}

\newcommand{\C}{\mathbb C}

\newcommand{\N}{\mathbb N}
\newcommand{\T}{\mathbb T}

\newcommand{\E}{\mathbb E}

\newcommand{\Pe}{\mathbb P}
\newcommand{\D}{\mathbb D}
\newcommand{\G}{\mathbb G}

\newcommand{\Pbar}{\overline{\mathbb P}}

\newcommand{\lm}{\lambda}

\newcommand{\HS}{\mathcal{H}}

\newcommand{\al}{\alpha}
\newcommand{\DC}{\overline{\mathbb{D}}}

\newcommand{\la}{\langle}
\newcommand{\ra}{\rangle}
\newcommand{\LS}{\mathscr{L}}
\newcommand{\KS}{\mathcal{K}}
\newtheorem{thm}{Theorem}[section]

\newtheorem{lem}[thm]{Lemma}

\newtheorem{prop}[thm]{Proposition}
\numberwithin{equation}{section} \theoremstyle{definition}
\newtheorem{defn}[thm]{Definition}

\def\textmatrix#1&#2\\#3&#4\\{\bigl({#1 \atop #3}\ {#2 \atop #4}\bigr)}
\def\dispmatrix#1&#2\\#3&#4\\{\left({#1 \atop #3}\ {#2 \atop #4}\right)}

\begin{document}
	
	\title[A Toeplitz Corona theorem for the pentablock]{A Toeplitz corona theorem for the pentablock and applications} 
	\author{SOURAV PAL AND NITIN TOMAR}
	
	\address[Sourav Pal]{Mathematics Department, Indian Institute of Technology Bombay,
		Powai, Mumbai - 400076, India.} \email{sourav@math.iitb.ac.in}
	
	\address[Nitin Tomar]{Mathematics Department, Indian Institute of Technology Bombay, Powai, Mumbai-400076, India.} \email{tomarnitin414@gmail.com}

	

	\begin{abstract}
We state and prove a Toeplitz corona theorem for the pentablock $\mathbb{P}$, a domain in $\mathbb{C}^3$ given by
\[
\mathbb{P}=\{(a_{21}, \text{tr}(A), \det(A)) \in \mathbb C^3 : A=[a_{ij}] \in M_2(\mathbb C), \|A\|<1\}.
\]
By two different applications of this theorem, we obtain a few new characterizations in the Toeplitz corona theorems for the bidisc and the symmetrized bidisc.
	\end{abstract}	
	
	\maketitle

	\section{Introduction}
	
	\noindent Throughout the paper, all operators are bounded linear maps acting on complex separable Hilbert spaces. For Hilbert spaces $\HS$ and $\mathcal{L}$, the space of all operators from $\HS$ to $\mathcal{L}$ is denoted by $\mathcal{B}(\HS, \mathcal{L})$ with $\mathcal{B}(\HS)=\mathcal{B}(\HS, \HS)$. Let $\C, \D$ and $\T$ be the complex plane, the unit disc and the unit circle in the complex plane, respectively, with center at the origin. For a commuting tuple of operators $\underline{T}=(T_1, \dotsc, T_d)$, we denote by $\sigma_T(\underline{T})$ its Taylor joint spectrum. The space of all holomorphic functions on a domain $\Omega \subseteq \C^d$ is denoted by $\text{Hol}(\Omega)$. 
	
	\smallskip 
	
	The corona problem originated in the study of maximal ideal spaces of Banach algebras of holomorphic functions. The classical setting is the Banach algebra $H^\infty(\D)$ of all bounded holomorphic functions on the unit disc $\D$, equipped with the supremum norm $\|.\|_{\infty, \D}$. The maximal ideal space $M_{H^\infty(\D)}$ is the set of all nonzero linear functionals on $H^\infty (\D)$, which is contained in the unit ball of the dual space of $H^\infty(\D)$. In fact, $M_{H^\infty(\D)}$ endowed with weak $*$-topology is a compact Hausdorff space. For each $w \in \D$, the point evaluation functional $e_w: f\mapsto f(w)$ for $f \in H^\infty(\D)$ defines an element of $M_{H^\infty(\D)}$ and thus, $\D$ becomes a subset of $M_{H^\infty(\D)}$. The \textit{corona} is a set defined as the complement of the weak $*$-closure of $\D$ in $M_{H^\infty(\D)}$. The corona problem on $\D$ asks whether this set is empty. Some routine arguments from Banach algebra theory show that $H^\infty (\D)$ has no corona if and only if the following holds: 
	  given $\varphi_1, \dotsc, \varphi_n \in H^\infty(\D)$ satisfying
	 \begin{align}\label{eqn_101}
	 \sup_{z \in\D}\left[|\varphi_1(z)|^2+\dotsc+|\varphi_n(z)|^2\right] \geq \epsilon^2
	 \end{align}
for some $\epsilon>0$, there exist functions $f_1, \dotsc, f_n$ in $H^\infty(\D)$ such that $\varphi_1f_1+\dotsc+\varphi_n f_n=1$. Carleson \cite{Carleson} proved that the aforementioned equivalent criterion holds and so, $H^\infty(\D)$ has no corona. Arveson \cite{Arveson_TC} proved an operator-theoretic analogue of this result, which is commonly referred to as the \textit{Toeplitz corona theorem} on $\D$, where the condition \eqref{eqn_101} is replaced by the condition
	 $\displaystyle 
	 T_{\varphi_1}T_{\varphi_1}^*+\dotsc+T_{\varphi_n}T_{\varphi_n}^* \geq \epsilon^2I.
	 $
	Evidently, $T_{\varphi_i}$ is the Toeplitz operator with symbol $\varphi_i$ for $1 \leq i \leq n$. Subsequently, Nagy and Foias \cite{SNF76} obtained a more general operator-valued version from which Toeplitz corona theorem follows as a special case. Later, Schubert \cite{Sch78} gave a shorter proof to the Toeplitz corona theorem together with explicit bounds on the norms of $f_1, \dotsc, f_n$ in $H^\infty(\D)$ such that $\varphi_1f_1+\dotsc+\varphi_nf_n=1$.  An interested reader can see \cite{Coburn, Gamelin, Hormander, Rosenblum} and the references therein to witness various generalizations of the Toeplitz corona theorem on $\D$. Almost three decades ago, Agler and McCarthy \cite{Agler_McCarthyI} established a Toeplitz corona theorem on the bidisc $\D^2$. Amar \cite{Amar} proved the same theorem on the Euclidean unit ball $\mathbb{B}_n=\{z \in \C^n : \|z\|<1\}$ and the polydisc $\D^n\subset \C^n$ in terms of measures supported on the distinguished boundaries $\partial \mathbb{B}_n=\{z \in \C^n : \|z\|=1\}$ and $n$-torus $\T^n$, respectively. The study of Toeplitz corona problems on domains such as symmetrized bidisc and tetrablock (defined below) have also attracted considerable attentions in recent past, e.g., see \cite{Tirtha_Sau, Jain}. 
	 
	\smallskip 

	In this article, we study the Toeplitz corona theorem on three domains the bidisc $\D^2$, the symmetrized bidisc $\mathbb G_2$ and the pentablock $\Pe$ (defined below). The domains $\mathbb G_2, \Pe$ arise naturally in the context of $\mu$-synthesis problem that is originated in the control theory, e.g., see \cite{Doyle}. In control theory, the structured singular value $\mu_E$ for a linear subspace $E \subseteq M_d(\C)$ is defined as
	\[
	\mu_E(A)=1\slash \inf\{\|X\|: X \in E , \ \det(I-AX)=0\} \quad (A \in M_d(\C)).
	\]
	Given distinct points $\alpha_1, \dots, \alpha_n$ in $\D$ and matrices $B_1, \dots, B_n \in M_d(\C)$, the $\mu$-synthesis problem seeks the existence of an analytic map $F : \D \to M_d(\C)$ satisfying $F(\alpha_i)=B_i$, $1\leq i \leq n$ together with the condition that $\mu_E(F(\lambda)) \leq 1$ for all $\lambda \in \D$. How different cases of the $\mu$-synthesis problem induce different domains such as symmetrized bidisc, tetrablock, pentablock are explicitly shown in \cite{AglerYoung}, \cite{Abouhajar} and \cite{AglerIV}, respectively. These three domains are defined as follows:
		\begin{align*}
		\qquad	\G_2&=\{(\lm_1+\lm_2, \lm_1\lm_2) \in \C^2: |\lm_1|, |\lm_2|<1 \},\\
		\qquad 	\E&=\left\{(z^{(1)}, z^{(2)}, z^{(3)}) \in \C^3 : 1-z^{(1)}\al_1-z^{(2)}\al_2+z^{(3)}\al_1\al_2 \neq 0 \ \text{for all} \ \al_1, \al_2 \in \overline{\D}\right\}, \\
		\qquad 	\Pe&=\{(a_{21}, \text{tr}(A), \det(A)) \in \mathbb C^3 : A=[a_{ij}] \in M_2(\mathbb C), \|A\|<1\}.
	\end{align*}
Several aspects of function theory on these domains have been studied extensively. For example, Toeplitz corona theorem on $\G_2$ was proved in \cite{Tirtha_Sau}, while the corresponding theorem on $\E$ was obtained recently in \cite{Jain}. In this direction, we prove Toeplitz corona theorem on $\Pe$. The present work builds upon \cite{Pal2026}, where realization, interpolation and extension theorems were established on $\Pe$. To do so, the authors introduced in \cite{Pal2026} the class $\mathfrak{M}_\Pe$ consisting of all commuting operator triples $\underline{T}=(T_1, T_2, T_3)$ satisfying
$\|T_2\|<2, \|(2\al T_3-T_2)(2-\al T_2)^{-1}\|<1$ and $\|(1-|\al|^2)T_1(I-\al T_2-\al^2T_3)^{-1}\|<1$ for all $\al \in \DC$. Also, $\sigma_T(\underline{T}) \subseteq \Pe$ for all $\underline{T} \in \mathfrak{M}_\Pe$ yielding a functional calculus $f(\underline{T})$ for $f \in \text{Hol}(\Pe)$. The Schur-Agler class $SA(\Pe)$ of $\Pe$ is defined as
\[
SA(\Pe)=\left\{f \in \text{Hol}(\Pe) : \|f(\underline{T})\| \leq 1 \ \text{for all} \ \underline{T} \in \mathfrak{M}_\Pe\right\}.
\]
The realization theorem on $\Pe$ obtained in \cite{Pal2026} provides characterizations of functions in $SA(\Pe)$. The proof of Toeplitz corona theorem on $\Pe$ capitalizes a vector-valued analog of realization theorem for $\Pe$. For this reason, we present in Section \ref{sec_02} the vector-valued analogs of all necessary results from \cite{Pal2026}. Note that proofs to these vector-valued results can be achieved easily if one follows the same line of arguments as in the scalar-valued cases, \cite{Pal2026}. Thus, mostly we shall refrain from writing (trivial) proofs in Section \ref{sec_02}. In Section \ref{sec_03}, we state and prove the Toeplitz corona theorem on $\Pe$. An important geometric property of $\Pe$ is that both $\G_2$ and $\D^2$ admit holomorphic embeddings into $\Pe$, which follows from the result stated below.

\begin{thm}[\cite{Pal_penta}, Section 3]\label{thm_connect_P}
	A pair $(z^{(1)}, z^{(2)}) \in \D^2$ if and only if $(z^{(1)}, 0, z^{(2)}) \in \Pe$. Also, $(z^{(1)}, z^{(2)}) \in \G_2$ if and only if $(0, z^{(1)}, z^{(2)}) \in \Pe$.
\end{thm}
Capitalizing on Theorem \ref{thm_connect_P}, the authors of \cite{Pal2026} recovered the realization, interpolation and extension theorems on $\D^2$ and $\G_2$ within the pentablock framework. In the same spirit, we present alternative characterizations for the Toeplitz corona theorems on $\D^2$ and $\G_2$ as consequences of the corresponding theorem on $\Pe$. This shows that the function theories on $\D^2$ and $\G_2$ can be studied in a unified manner through the function theory on $\Pe$.
	 
	\section{Preliminaries}\label{sec_02}
	
	\noindent In this section, we recall some definitions and results from the literature that will be used in the sequel. We begin with the characterization of the pentablock  from \cite{AglerIV} given by
	\[
	\mathbb P= \left\{(z^{(1)}, z^{(2)}, z^{(3)}) \in \mathbb{C} \times \mathbb{G}_2 :
	\sup_{\alpha \in \mathbb{D}}
	\left|\frac{(1-|\al|^2)z^{(1)}}{1- z^{(2)}\al+z^{(3)}\al^2}\right|<1 \right\},
	\]
	where $\G_2=\{(\lm_1+\lm_2, \lm_1\lm_2): \lm_1, \lm_2 \in \D\}$ is the symmetrized bidisc. It was proved in \cite{Agler2004_II} that $(z^{(2)}, z^{(3)}) \in \G_2$ if and only if $|z^{(2)}|<2$ and $|2\al z^{(3)}-z^{(2)}|<|2-\al z^{(2)}|$ for every $\al \in \DC$. Consequently, a point $z=(z^{(1)}, z^{(2)}, z^{(3)}) \in \Pe$ if and only if 
	\begin{equation}\label{eqn_P}
		|z^{(2)}|<2, \quad |\Phi_\alpha(z^{(2)}, z^{(3)})|<1 \quad \text{and} \quad |\psi_{\al}(z^{(1)}, z^{(2)}, z^{(3)})|<1
	\end{equation}
	for all $\al \in \DC$, where 
	$\displaystyle \Phi_\alpha(z^{(2)}, z^{(3)})=\frac{2\al z^{(3)}-z^{(2)}}{2-\al z^{(2)}}$ and $\displaystyle \psi_{\al}(z^{(1)}, z^{(2)}, z^{(3)})=\frac{(1-|\al|^2)z^{(1)}}{1- z^{(2)}\al+z^{(3)}\al^2}$.
	Using the above description of $\Pe$, we define for $z = (z^{(1)}, z^{(2)}, z^{(3)}) \in \Pe$ the following functions:
	\begin{align}\label{eqn_J(z)}
		\mathrm{J}(z), \ \mathrm{j}(z): \DC \to \C, \quad \mathrm{J}(z)(\al)=\psi_{\al}(z^{(1)}, z^{(2)}, z^{(3)}) \quad \text{and} \quad \mathrm{j}(z)(\al) =\Phi_\al(z^{(2)}, z^{(3)}).
	\end{align}
	Clearly, $\mathrm{J}(z)$ and $\mathrm{j}(z)$ are continuous functions on $\DC$, and $\|\mathrm{J}(z)\|_{\infty, \DC}, \|\mathrm{j}(z)\|_{\infty, \DC}<1$ for $z \in \Pe$. The authors of \cite{Pal2026} introduced an operator-theoretic analog of the inequalities provided in \eqref{eqn_P}. To see this, we first consider the class 
	$
	\mathfrak{M}_{\G_2}=\left\{(T_2, T_3): \|T_2\|<2, \|\Phi_\al(T_2, T_3)\|<1 \ \text{for all} \ \al \in \DC\right\}.
	$
Note that $\sigma_T(T_2, T_3) \subseteq \G_2$ and so, holomorphic functional calculus yields an operator $f(T_2, T_3)$ for every $f \in \text{Hol}(\G_2)$. The Schur-Agler class $SA(\G_2)$ for $\G_2$ is the collection of all $f \in \text{Hol}(\G_2)$ such that $\|f(\underline{T})\| \leq 1$ for every $\underline{T} \in \mathfrak{M}_{\G_2}$. For the pentablock, the authors of \cite{Pal2026} considered the class of commuting triples given by
	\[
	\mathfrak{M}_{\Pe}=\left\{(T_1, T_2, T_3): (T_2, T_3) \in \mathfrak{M}_{\G_2} \ \text{and} \ \|\psi_{\al}(T_1, T_2, T_3)\|<1 \ \text{for every $\al \in \DC$} \right\}.
	\]
	It was proved in \cite{Pal2026} that $\sigma_T(T_1, T_2, T_3) \subseteq \Pe$ for all $(T_1, T_2, T_3) \in \mathfrak{M}_\Pe$. Consequently, we have a functional calculus $f(\underline{T})$ for all $f \in \mathrm{Hol}(\mathbb{P})$ and $\underline{T} \in \mathfrak{M}_\mathbb{P}$. So, one can consider the class given by
	\[
	SA(\Pe)=\left\{f \in \text{Hol}(\Pe) : \|f(\underline{T})\| \leq 1 \ \text{for all} \ \underline{T} \in \mathfrak{M}_\Pe\right\},
	\]
	which is called the Schur-Agler class for $\Pe$. One of the key results in \cite{Pal2026} is a realization theorem on $\Pe$, which is precisely the characterization of functions in $SA(\Pe)$. To do so, the authors \cite{Pal2026} introduced the notion of admissible kernels associated with $\Pe$. In this direction, we present the vector-valued analog of admissible kernels on $\Pe$. Let $\LS$ be a Hilbert space, and let $F$ be a subset of $\Pe$. A $\mathcal{B}(\LS)$-valued weak kernel of $F$ is a positive semi-definite function $k: F \times F \to \mathcal{B}(\LS)$, that is, 
	$
	\overset{n}{\underset{i, j=1}{\sum}}\la k(z_i, z_j)v_j, v_i\ra_{\LS} \geq 0 
	$
	for every finite set $\{z_1, \dotsc, z_n\} \subset F$ and vectors $v_1, \dotsc, v_n \in \LS$. In addition, if $k(z, z) \ne 0$ for all $z \in F$, we say that $k$ is a $\mathcal{B}(\LS)$-valued kernel. A $\mathcal{B}(\LS)$-valued kernel $k$ is said to be \textit{admissible} if the maps
	\begin{align*}
		& (z, w) \mapsto \left(1-\frac{z^{(2)}\overline{w}^{(2)}}{4}\right)k(z, w), \\
		& (z, w) \mapsto \left(1-\Phi_\al(z^{(2)}, z^{(3)})\overline{\Phi_\al( w^{(2)}, w^{(3)})}\right)k(z, w) \qquad \text{and} \\
		& (z, w) \mapsto \left(1-\psi_{\al}(z^{(1)}, z^{(2)}, z^{(3)})\overline{\psi_{\al}(w^{(1)}, w^{(2)}, w^{(3)})}\right)k(z, w)
	\end{align*}
	are positive semi-definite for all $\al \in \DC$. The class of scalar-valued admissible kernels on $\Pe$ is denoted by $AK(\Pe)$. A map $\xi: F \times F \to \mathcal{B}(C(\DC), \mathcal{B}(\LS))$ is called \textit{completely positive kernel} if for every $n \in \N, \{v_1, \dotsc, v_n\} \subset \LS, \{z_1, \dotsc, z_n\} \subset F$ and $\{h_1, \dotsc, h_n\} \subset C(\DC)$, we have
	\[
	\overset{n}{\underset{i, j=1}{\sum}}\la \xi(z_i, z_j)(h_i\overline{h}_j)v_j, v_i\ra_\LS \geq 0.
	\]
If $\LS=\C$, then $\mathcal{B}(C(\DC), \LS)$ is the dual space $C(\DC)^*$ of $C(\DC)$. In this case,  a completely positive kernel $\xi: F \times F \to C(\DC)^*$ is simply called a \textit{positive kernel}. In order to prove a Toeplitz corona theorem on $\Pe$, we need an operator-valued analogue of the realization theorem on $\Pe$ (see Theorem 2.7 in \cite{Pal2026}). Also, we shall require operator-valued analogues of some other scalar-valued results from \cite{Pal2026}. As we have already mentioned, we shall skip the proofs in most of cases as they are similar to corresponding the scalar-valued results. We begin with the following proposition which is an operator theoretic version of Proposition 2.2 in \cite{Pal2026}.

\begin{prop}\label{prop_prelim_I_P}
	Let $F$ be a subset of $\Pe$. If $\xi: F \times F \to  \mathcal{B}(C(\DC), \mathcal{B}(\LS))$ is a completely positive kernel, then there exist a Hilbert space $\HS$, a function $L: F \to \mathcal{B}(C(\DC), \mathcal{B}(\HS, \LS))$ and a unital $*$-representation $\rho: C(\DC) \to \mathcal{B}(\HS)$ such that
	\[
	\xi(z, w)(f\overline{g})=L(z)(f)(L(w)(g))^* \quad \text{and} \quad L(z)(fg)^*=\rho(f)^*L(z)(g)^*
	\]	
	for all $f, g \in C(\DC)$ and $z, w  \in F$. 
\end{prop}	
Let $k_1, k_2$ be $\mathcal{B}(\LS)$-valued functions on $\Pe \times \Pe$. Following the notations due to Agler and McCarthy as in Definition 11.25 of \cite{Agler_McCarthy}, we set $k_1\oslash k_2: \Pe \times \Pe \to \mathcal{B}(\LS \otimes \LS)$ as $(k_1 \oslash k_2)(z, w)=k_1(z, w)\otimes k_2(z, w)$ for $z, w \in \Pe$. Our next result is a vector-valued analog of Theorem 2.3 in \cite{Pal2026} and again we skip its proof.	
	\begin{thm}\label{thm_sa_P}
	Let $g: \Pe \times \Pe \to \mathcal{B}(\LS)$ be a self-adjoint function, that is, $g(z, w)=g(w, z)^*$. If $g\oslash k$ is positive semi-definite for any $\mathcal{B}(\LS)$-valued admissible kernel $k$ on $\Pe$, then there exist completely positive kernels $\xi, \nabla : \Pe \times \Pe \to \mathcal{B}(C(\DC), \mathcal{B}(\LS))$ and a weak $\mathcal{B}(\LS)$-valued kernel $\delta$ on $\Pe$ such that
	\[
	g(z, w)=\xi(z, w)(1-\mathrm{J}(z)\overline{\mathrm{J}(w)})+\nabla(z, w)(1-\mathrm{j}(z)\overline{\mathrm{j}(w)})+(1-(z^{(2)}\overline{w}^{(2)}\slash 4))\delta(z, w)
	\]
	for all $z, w \in \Pe$. 
\end{thm}		

To present the statement of operator-valued realization theorem for $\Pe$, we need to define a vector-valued analog of the Schur-Agler class $SA(\Pe)$. To do so, we recall certain notions of analytic functional calculus for which we refer to \cite{Ambrozie, Esch}. For a domain $\Omega$ in $\C^d$ and Hilbert spaces $\HS_1, \HS_2$, let $f: \Omega \to \mathcal{B}(\HS_1, \HS_2)$ be an analytic function. Suppose $\underline{T}=(T_1, \dotsc, T_d)$ is a commuting tuple of operators on a Hilbert space $\HS$ with $\sigma_T(\underline{T}) \subseteq \Omega$. Then $f(\underline{T}): \HS \otimes \HS_1 \to \HS \otimes \HS_2$ is an bounded linear map defined by analytic functional calculus. We discuss a description of $f(\underline{T})$ presented in \cite{Esch}. Given an open neighbourhood $U$ of $\sigma_T(\underline{T})$, there is a unique continuous linear map $f \mapsto f(\underline{T})$ from $\text{Hol}(U, \mathcal{B}(\HS_1, \HS_2)) \equiv \text{Hol}(U)\otimes \mathcal{B}(\HS_1, \HS_2)$ into $\mathcal{B}(\HS \otimes \HS_1, \HS \otimes \HS_2)$ taking $f \otimes A$ into $f(\underline{T}) \otimes A$ for $f \in \text{Hol}(U)$ and $A \in \mathcal{B}(\HS_1, \HS_2)$. Consequently, we have 
	\[
	\la f(\underline{T})(h_1 \otimes x), h_2 \otimes y\ra=\la f_{x, y}(\underline{T})h_1, h_2 \ra 
	\]
	for every $h_1, h_2 \in \HS, x \in \HS_1$ and $y \in \HS_2$, where $f_{x, y}: \Omega \to \C$ is the scalar-valued holomorphic map defined by $f_{x, y}(z)=\la f(z)x, y\ra$. For further clarity, we consider the finite-dimensional case here. If $\HS_1=\C^q, \HS_2=\C^p$ and $f=[f_{ij}]_{i, j}$, then we have the isometric isomorphisms $\HS \otimes \HS_1 \equiv \HS^q$ and $\HS \otimes \HS_2 \equiv \HS^p$ and so, $f(\underline{T})=[f_{ij}(\underline{T})]_{i, j} \in \mathcal{B}(\HS^q, \HS^p)$ with $f_{ij}(\underline{T}) \in \mathcal{B}(\HS)$. A vector-valued analog of Schur-Agler type class of $\Pe$ can be defined now in the following way. 
	
	\begin{defn}
	For Hilbert spaces $\LS_1, \LS_2$, the $\mathcal{B}(\LS_1, \LS_2)$-valued Schur-Agler class $SA_\Pe(\LS_1, \LS_2)$ is the class of all holomorphic functions $f: \Pe \to \mathcal{B}(\LS_1, \LS_2)$ such that $\|f(\underline{T})\| \leq 1$ for every $\underline{T} \in \mathfrak{M}_\Pe$. 
	\end{defn}

For a holomorphic map $f: \Pe \to \mathcal{B}(\LS_1, \LS_2)$, and $\{x_1,\dotsc, x_m\}\subseteq \LS_1$ and $\{y_1,\dotsc, y_n\}\subseteq \LS_2$, we say that the matrix-valued function
$
F=[f_{x_i, y_j}]_{j,i}
$
belongs to the unit ball of $M_{n,m}(SA(\Pe))$ if $\|F(\underline{T})\| \leq 1$ for all $\underline{T} \in \mathfrak{M}_\Pe$. The operator-valued Schur-Agler class $SA_\Pe(\LS_1, \LS_2)$ can be completely characterized in terms of the Schur-Agler class $SA(\Pe)$ as the following lemma shows.

\begin{lem}\label{lem_201}
	Let $\LS_1$ and $\LS_2$ be Hilbert spaces, and let
	$f:\Pe \to \mathcal{B}(\LS_1,\LS_2)$ be a holomorphic map. Then  $f\in SA_{\Pe}(\LS_1,\LS_2)$ if and only if for all finite orthonormal subsets $\{e_1,\dotsc, e_m\}\subseteq \LS_1$ and $\{e'_1,\dotsc, e'_n\}\subseteq \LS_2$, the matrix-valued function
$
		[f_{e_i, e'_j}]_{j,i}
$
belongs to the unit ball of $M_{n,m}(SA(\Pe))$, where
		$
		f_{x,y}(z)=\la f(z)(x),y\ra_{\LS_2}$ for $x \in \LS_1, y \in \LS_2$ and $z \in \Pe$.
		
\end{lem}

\begin{proof}
	Suppose $f\in SA_{\Pe}(\LS_1,\LS_2)$. Let $\underline{T}\in \mathfrak{M}_{\Pe}$ be acting on a Hilbert space $\HS$. Take finite orthonormal subsets $\{e_1,\dotsc,e_m\}$ and $\{e'_1,\dotsc, e'_n\}$ in  $\LS_1$ and $\LS_2$, respectively. Define $F=[f_{e_i, e'_j}]_{j,i}$. Let $h_1, \dotsc, h_m, h'_1, \dotsc, h'_n \in \HS$. 	Then
		\begin{align*}
		\la F(\underline{T})(h_1,\dotsc,h_m),(h'_1,\dotsc,h'_n)\ra 
=
		\sum_{j=1}^n
		\left\la
		\sum_{i=1}^m f_{e_i, e'_j}(\underline{T})h_i,
		h'_j
		\right\ra 
&	=	\sum_{i=1}^m\sum_{j=1}^n
		\la f(\underline{T})(h_i\otimes e_i),h'_j\otimes e'_j\ra \\
		&=
		\left\la
		f(\underline{T})
		\left(
		\sum_{i=1}^m h_i\otimes e_i
		\right),
		\sum_{j=1}^n h'_j\otimes e'_j
		\right\ra.
	\end{align*}
Since $\{e_i : 1 \leq i \leq m\}$ and $\{e'_j: 1 \leq j \leq n\}$ are orthonormal subsets, it follows that
	\begin{align*}
		\left\|
		\sum_{i=1}^m h_i\otimes e_i
		\right\|^2
		=
		\sum_{i,r=1}^m
		\la h_i,h_r\ra
		\la e_i,e_r\ra 
		=
		\sum_{i=1}^m \|h_i\|^2 \quad \text{and similarly,} \quad 	\left\|
		\sum_{j=1}^n h'_j\otimes e'_j
		\right\|^2
		=
		\sum_{j=1}^n \|h'_j\|^2.
	\end{align*}
Therefore,
	\begin{align*}
	\left|
	\la F(\underline{T})(h_1,\dotsc,h_m),(h'_1,\dotsc,h'_n)\ra
	\right| 
	 \leq
	\left\|
	\sum_{i=1}^m h_i\otimes e_i
	\right\|
	\left\|
	\sum_{j=1}^n h'_j\otimes e'_j
	\right\|
		& = 
		\|(h_1,\dotsc,h_m)\|
		\,
		\|(h'_1,\dotsc,h'_n)\|.
	\end{align*}
	Hence, $\|F(\underline{T})\|\le 1$ and so, $F$ is in the unit ball of $M_{n,m}(SA(\Pe))$. 
	
	\smallskip 
	
	To see the converse, let
	$\underline{T}\in \mathfrak{M}_{\Pe}$ be acting on a Hilbert space $\HS$. We show that
	$\|f(\underline{T})\|\le 1$. Take $u\in \HS\otimes \LS_1$ and $v\in \HS\otimes \LS_2$. 
	Since finite sums of elementary tensors are dense in $\HS\otimes \LS_1$ and $\HS\otimes \LS_2$, and after replacing the vectors appearing in the tensor representations of $u$ and $v$ by orthonormal bases for their spans, it suffices to consider 
	\[
	u=\sum_{i=1}^m h_i\otimes e_i \quad \text{and}	\quad v=\sum_{j=1}^n h'_j\otimes e'_j,
	\] 
	where $\{e_1,\dotsc,e_m\}\subseteq \LS_1$ and $\{e'_1,\dotsc, e'_n\}\subseteq \LS_2$
	are orthonormal sets. Define $F=[f_{e_i,e'_j}]_{j,i}$. By given hypothesis, $F$ belongs to the unit ball of $M_{n,m}(SA(\Pe))$. Then
	\begin{align*}
		\la f(\underline{T})u,v\ra
		=
		\sum_{i=1}^m\sum_{j=1}^n
		\la f(\underline{T})(h_i\otimes e_i),h'_j\otimes e'_j\ra 
		=
		\sum_{i, j}
		\la f_{e_i, e'_j}(\underline{T})h_i,h'_j\ra 
		=
		\la
		F(\underline{T})(h_1,\dotsc,h_m),
		(h'_1,\dotsc,h'_n)
		\ra.
	\end{align*}
Since $\|F(\underline{T})\|\le 1$, we have that $|\la f(\underline{T})u,v\ra| \leq 		\|(h_1,\dotsc,h_m)\|
\,
\|(h'_1,\dotsc,h'_n)\|=\|u\| \|v\|$. By limiting arguments, it follows that $\|f(\underline{T})\| \leq 1$ and thus, $f\in SA_{\Pe}(\LS_1,\LS_2)$.
\end{proof}

A realization theorem on $\Pe$, which provides a characterization of functions in the Schur-Agler class $SA(\Pe)$, was established in Theorem 2.7 of \cite{Pal2026}. We now extend this result to vector-valued setting and also skip its proof as it follows the same arguments as in the proof of Theorem 2.7 in \cite{Pal2026}.
	\begin{thm}\label{thm_realization_P}
	Let $\LS_1$ and  $\LS_2$ be Hilbert spaces. For a function $f: \Pe \to \mathcal{B}(\LS_1, \LS_2)$, the following statements are equivalent: 
	\begin{enumerate}[leftmargin=*]
		\item[$(1)$] $f \in SA_\Pe(\LS_1, \LS_2)$;
		\item[$(2)$] $(z, w) \mapsto (I_{\LS_2}-f(z)f(w)^*)\oslash k(z, w)$ is a positive semi-definite function for all $\mathcal{B}(\LS_2)$-valued admissible kernel $k$ on $\Pe$;
		\item[$(3)$] there exist completely positive kernels $\xi, \nabla: \Pe \times \Pe \to \mathcal{B}(C(\DC), \mathcal{B}(\LS_2))$ and a $\mathcal{B}(\LS_2)$-valued weak kernel $\delta: \Pe \times \Pe \to \mathcal{B}(\LS_2)$ such that
		\begin{align*}
			I_{\LS_2}-f(z)f(w)^*
			=\xi(z, w)(1-\mathrm{J}(z)\overline{\mathrm{J}(w)})+\nabla(z, w)(1-\mathrm{j}(z)\overline{\mathrm{j}(w)})+\left(1-\frac{z^{(2)}\overline{w}^{(2)}}{4}\right)\delta(z, w).
		\end{align*}
		\item[$(4)$] there exist Hilbert spaces $\HS_1, \HS_2, \HS_3$,  two unital $*$-representations $\rho_1: C(\DC) \to \mathcal{B}(\HS_1), \rho_2: C(\DC) \to \mathcal{B}(\HS_2)$ and a unitary $V=\begin{bmatrix}
			A & B \\
			C & D \\
		\end{bmatrix}: \LS_1 \oplus \HS \to \LS_2 \oplus \HS$, where $\HS=\HS_1 \oplus \HS_2 \oplus \HS_3$ such that 
		\[
		f(z)=A+B\begin{bmatrix} \rho_1(\mathrm{J}(z)) & 0 & 0\\ 
			0 & \rho_2(\mathrm{j}(z)) & 0\\
			0 & 0 & (z^{(2)}\slash 2)I_{\HS_3}
		\end{bmatrix} \left(I-D\begin{bmatrix} \rho_1(\mathrm{J}(z)) & 0 & 0\\ 
		0 & \rho_2(\mathrm{j}(z)) & 0\\
		0 & 0 & (z^{(2)}\slash 2)I_{\HS_3}
		\end{bmatrix}\right)^{-1}C.
		\]  
	\end{enumerate}
\end{thm}

\section{Toeplitz corona theorem on the pentablock}\label{sec_03}

\noindent In this section, we first state and prove a Toeplitz corona theorem on the pentablock $\Pe$. Further, we obtain new characterizations for Toeplitz corona theorems on $\D^2$ and $\G_2$ in terms of admissible kernels and completely positive kernels on $\Pe$. We begin with an important theorem.

\begin{thm}\label{thm_TC_P}
	Let $\LS_1, \LS_2$ and $\LS_3$ be Hilbert spaces. For functions $\Upsilon_{12}: \Pe \to \mathcal{B}(\LS_1, \LS_2)$ and $\Upsilon_{32}: \Pe \to \mathcal{B}(\LS_3, \LS_2)$, the following are equivalent: 
	\begin{enumerate}[leftmargin=*]
		\item[$(1)$] there exists $\Upsilon_{31} \in SA_\Pe(\LS_3, \LS_1)$ such that $\Upsilon_{12}(z)\Upsilon_{31}(z)=\Upsilon_{32}(z)$ for every $z \in \Pe$;
		\item[$(2)$] for every finite subset $F=\{z_1, \dotsc, z_n\} \subset \Pe$,
		\[
		\begin{bmatrix}
			\left(\Upsilon_{12}(z_i)\Upsilon_{12}(z_j)^*-\Upsilon_{32}(z_i)\Upsilon_{32}(z_j)^*\right)\otimes k(z_i, z_j)
		\end{bmatrix}_{i, j=1}^n \geq 0
		\]
	for every $\mathcal{B}(\LS_2)$-valued admissible kernel $k$ on $\Pe$;
	\item[$(3)$] there exist completely positive kernels $\xi, \nabla: \Pe \times \Pe \to \mathcal{B}(C(\DC), \mathcal{B}(\LS_2))$ and a weak kernel $\delta: \Pe \times \Pe \to \mathcal{B}(\LS_2)$ such that for all $z, w \in \Pe$,
	\begin{align*}
	&\Upsilon_{12}(z)\Upsilon_{12}(w)^*-\Upsilon_{32}(z)\Upsilon_{32}(w)^*\\
	&=\xi(z, w)(1-\mathrm{J}(z)\overline{\mathrm{J}(w)})+\nabla(z, w)(1-\mathrm{j}(z)\overline{\mathrm{j}(w)})+\left(1-\frac{z^{(2)}\overline{w}^{(2)}}{4}\right)\delta(z, w),
	\end{align*}
where the maps $z\mapsto \mathrm{J}(z)$ and $z\mapsto \mathrm{j}(z)$ are as in \eqref{eqn_J(z)}.	 	
	\end{enumerate}
\end{thm}	

\begin{proof}	$(1) \implies (2)$. Let $k: \Pe \times \Pe \to \mathcal{B}(\LS_2)$ be an admissible kernel, and let $F=\{z_1, \dotsc, z_n\}$ be a subset of $\Pe$. For $w \in \Pe$ and $h \in \LS_2$, we define $k_wh: \Pe \to \LS_2$ as $k_wh(z)=k(z, w)h$. Let 
\[
\HS(k)=\overline{\text{span}}\left\{k_{z}h: z \in \Pe, h \in \LS_2\right\},
\]
which is the reproducing kernel Hilbert space associated with the operator-valued kernel $k$, and is equipped with the inner product $\la k_zh, k_w h'\ra_{\HS(k)}=\la k(w, z)h, h'\ra_{\LS_2}$. Consider the commuting triple $\underline{M}=(M_{z^{(1)}}, M_{z^{(2)}}, M_{z^{(3)}})$ consisting of the coordinate multiplication operators on $\HS(k)$. It is not difficult to see that  $M_{z^{(i)}}^*(k_zh)=k_z(\overline{z}^{(i)}h)$ for $ 1 \leq i \leq 3$. For every $z, w \in \Pe$ and $h, h' \in \LS_2$, we have
{\small 
\begin{enumerate}[leftmargin=*]
	\item[(i)] $\left\la \left(I-\frac{M_{z^{(2)}}M_{z^{(2)}}^*}{4}\right)k_zh, k_wh'\right\ra
	=\la k(w, z)h, h'\ra -\frac{1}{4}\left\la k_z(\overline{z}^{(2)}h), k_w(\overline{w}^{(2)}h') \right \ra 
	=\left(1-\frac{\overline{z}^{(2)}w^{(2)}}{4}\right)\la k(w, z)h, h'\ra$, \smallskip 
	
	\item[(ii)] $\left\la \left(I-\Phi_\al(M_{z^{(2)}}, M_{z^{(3)}})\Phi_\al(M_{z^{(2)}}, M_{z^{(3)}})^*\right)k_zh, k_wh'\right\ra=\left(1-\overline{\Phi_\al(z^{(2)}, z^{(3)})}\Phi_\al(w^{(2)}, w^{(3)})  \right)\la k(w, z)h, h'\ra$, \smallskip 
	
	 \item[(iii)] $\left\la \left(I-\Psi_\al(\underline{M})\Psi_\al(\underline{M})^*\right)k_zh, k_wh'\right\ra=\left(1-\overline{\psi_\al(z^{(1)}, z^{(2)}, z^{(3)})}\psi_\al(w^{(1)}, w^{(2)}, w^{(3)})  \right)\la k(w, z)h, h'\ra$.
\end{enumerate}}
\noindent Since $k$ is an admissible kernel, it now follows that for every $\al \in \DC$,
\begin{align}\label{eqn_201}
\|M_{z^{(2)}}\| \leq 2, \quad \|\Phi_\al(M_{z^{(2)}}, M_{z^{(3)}})\| \leq 1 \quad \text{and} \quad \|\psi_\al(M_{z^{(1)}}, M_{z^{(2)}}, M_{z^{(3)}})\| \leq 1
\end{align}
Associated with the set $F=\{z_1, \dotsc, z_n\}$, we consider $\HS_n(k)=\overline{\text{span}}\{k_{z_i}h: h \in \LS_2, 1 \leq i \leq n\}$, which is a subspace of $\HS(k)$. Since $M_{z^{(i)}}^*(k_{z_j}h)=k_{z_j}(\overline{z}_j^{(i)}h)$ for $1 \leq i \leq 3$ and $1 \leq j \leq n$, it follows that $\HS_n(k)$ is a joint invariant subspace for $\underline{M}^*=(M_{z^{(1)}}^*, M_{z^{(2)}}^*, M_{z^{(3)}}^*)$. Consider the operator triple $\underline{T}=(T_1, T_2, T_3)$ on $\HS_n(k)$ given by
\[
(T_1^*, T_2^*, T_3^*)=\left(M_{z^{(1)}}^*|_{\HS_n(k)}, M_{z^{(2)}}^*|_{\HS_n(k)}, M_{z^{(3)}}^*|_{\HS_n(k)}\right).
\]
We have by \eqref{eqn_201} that $\|T_2^*\| \leq 2, \|\Phi_\al(T_2^*, T_3^*)\| \leq 1$ and $\|\psi_\al(T_1^*, T_2^*, T_3^*)\| \leq 1$ for all $\al \in \DC$. Evidently $\Pe$ is a $(1, 1, 2)$-quasi-balanced, that is, $s\cdot z=(sz^{(1)}, sz^{(2)}, s^2z^{(3)}) \in \Pe$ for $0 \leq s \leq 1$ and $z=(z^{(1)}, z^{(2)}, z^{(3)}) \in \Pe$. Consequently, $\|sT_2^*\| < 2, \|\Phi_\al(sT_2^*, s^2T_3^*)\| < 1$ and $\|\psi_\al(sT_1^*, sT_2^*, s^2T_3^*)\| < 1$ for all $\al \in \DC$ and $s \in (0, 1)$. Thus, $s\cdot\underline{T}^*=(sT_1^*, sT_2^*, s^2T_3^*) \in \mathfrak{M}_\Pe$ for every $s \in (0, 1)$. Let $f: \Pbar \to \mathcal{B}(\LS_3, \LS_1)$ be a holomorphic map such that $f \in SA_\Pe(\LS_3, \LS_1)$. Consider the holomorphic function $\widehat{f}: \Pbar \to \mathcal{B}(\LS_1, \LS_3)$ defined as $\widehat{f}(z)=f(\overline{z})^*$ for $z \in \Pbar$. It is not difficult to see that $\widehat{f} \in SA_\Pe(\LS_1, \LS_3)$. By analytic functional calculus, it follows that $\widehat{f}(s\cdot \underline{T}^*) \in \mathcal{B}(\HS_n(k)\otimes \LS_1, \HS_n(k)\otimes \LS_3)$. For $u_i, u_j \in \LS_2, v_i \in \LS_3$ and $v_j \in \LS_1$ with $ 1\leq i, j \leq n$,   
\begin{align*}
	\la \widehat{f}(s\cdot \underline{T}^*)(k_{z_j}u_j\otimes v_j), k_{z_i}u_i\otimes v_i \ra 	
	&=\la \widehat{f}_{v_j, v_i}(s\cdot \underline{T}^*)k_{z_j}u_j, k_{z_i}u_i \ra \\
	&=\la k_{z_j}(\widehat{f}_{v_j, v_i}(s\cdot \overline{z}_j)u_j), k_{z_i}u_i\ra \quad [M_{z^{(i)}}^*(k_{z_j}h)=k_{z_j}(\overline{z}_j^{(i)}h) \ \text{for} \ 1 \leq i \leq 3]\\
	&=  \la k(z_i, z_j)(\widehat{f}_{v_j, v_i}(s\cdot \overline{z}_j)u_j), u_i\ra \\
	&=\widehat{f}_{v_j, v_i}(s\cdot \overline{z}_j) \la k(z_i, z_j)u_j, u_i\ra \\
	&=\la \widehat{f}(s\cdot \overline{z}_j)v_j, v_i \ra \la k(z_i, z_j)u_j, u_i\ra\\
	&=\la f(s\cdot z_j)^*v_j, v_i \ra \la k(z_i, z_j)u_j, u_i\ra\\
	&= \la k_{z_j}u_j \otimes f(s\cdot \overline{z}_j)^*v_j, k_{z_i}u_i \otimes v_i \ra
\end{align*}
and thus, we have for every $u \in \LS_2, v \in \LS_1$ and $1 \leq j \leq n$ that
\begin{align}\label{eqn_202}
\widehat{f}(s\cdot \underline{T}^*)(k_{z_j}u\otimes v)=k_{z_j}u\otimes f(s\cdot z_j)^*v.
\end{align} 
Since $\widehat{f} \in SA_\Pe(\LS_1, \LS_3)$ and $s\cdot \underline{T}^* \in \mathfrak{M}_\Pe$, it follows that $\|\widehat{f}(s\cdot \underline{T}^*)\| \leq 1$. For $\{u_i, v_i : 1 \leq i \leq n\} \subset \LS_2$,

\begin{align*}
0 & \leq \left\la \left(I_{\HS_n(k)\otimes \HS_1}-\widehat{f}(s\cdot \underline{T}^*)^*\widehat{f}(s\cdot \underline{T}^*)\right) \left[\overset{n}{\underset{j=1}{\sum}}k_{z_j}u_j\otimes \Upsilon_{12}(z_j)^*v_j  \right], \left[\overset{n}{\underset{i=1}{\sum}}k_{z_i}u_i\otimes \Upsilon_{12}(z_i)^*v_i  \right]\right\ra \\
&= \overset{n}{\underset{i, j=1}{\sum}}\left \la k_{z_j}u_j\otimes \Upsilon_{12}(z_j)^*v_j, k_{z_i}u_i\otimes \Upsilon_{12}(z_i)^*v_i \right \ra \\
& \quad -
 \overset{n}{\underset{i, j=1}{\sum}}\left \la \widehat{f}(s\cdot \underline{T}^*)\left[k_{z_j}u_j\otimes \Upsilon_{12}(z_j)^*v_j\right], \widehat{f}(s\cdot \underline{T}^*)\left[k_{z_i}u_i\otimes \Upsilon_{12}(z_i)^*v_i\right] \right \ra \\
&= \overset{n}{\underset{i, j=1}{\sum}}\left \la k_{z_j}u_j\otimes \Upsilon_{12}(z_j)^*v_j, k_{z_i}u_i\otimes \Upsilon_{12}(z_i)^*v_i \right \ra \\
& \quad   -
\overset{n}{\underset{i, j=1}{\sum}} \left \la k_{z_j}u_j\otimes f(s\cdot z_j)^*\Upsilon_{12}(z_j)^*v_j, k_{z_i}u_i\otimes f(s\cdot z_i)^*\Upsilon_{12}(z_i)^*v_i \right \ra \\
&= \overset{n}{\underset{i, j=1}{\sum}}\left \la k(z_i, z_j)u_j, u_i \right \ra \left \la \Upsilon_{12}(z_i)\Upsilon_{12}(z_j)^*v_j, v_i \right \ra \\
& \quad -
\overset{n}{\underset{i, j=1}{\sum}} \left \la k(z_i, z_j)u_j, u_i \right \ra \left \la f(s\cdot z_j)^*\Upsilon_{12}(z_j)^*v_j, f(s\cdot z_i)^*\Upsilon_{12}(z_i)^*v_i \right \ra \\
&= \overset{n}{\underset{i, j=1}{\sum}} \left \la k(z_i, z_j)u_j, u_i \right \ra
 \left \la \left(\Upsilon_{12}(z_i)\Upsilon_{12}(z_j)^*-\Upsilon_{12}(z_i)f(s\cdot z_i) f(s\cdot z_j)^*\Upsilon_{12}(z_j)^*\right)v_j,v_i \right \ra
\end{align*}
and taking the limit $s \to 1^-$, we have that
\begin{align}\label{eqn_203}
\overset{n}{\underset{i, j=1}{\sum}} \left \la k(z_i, z_j)u_j, u_i \right \ra
\left \la \left(\Upsilon_{12}(z_i)\Upsilon_{12}(z_j)^*-\Upsilon_{12}(z_i)f(z_i) f(z_j)^*\Upsilon_{12}(z_j)^*\right)v_j,v_i \right \ra \geq 0.
\end{align}
For $\Upsilon_{31} \in SA_\Pe(\LS_3, \LS_1)$ and $0<r<1$, we define $\Upsilon_{31}^{(r)}: \Pbar \to \mathcal{B}(\LS_3, \LS_1)$ as $\Upsilon_{31}^{(r)}(z)=\Upsilon_{31}(r\cdot z)$. Evidently, $r\cdot z \in \Pe$ for every $z \in \Pbar$ and $r \in (0, 1)$. Thus, $\Upsilon_{31}^{(r)}$ is a well-defined holomorphic map on $\Pbar$ that belongs to $SA_\Pe(\LS_3, \LS_1)$. By \eqref{eqn_203}, it follows that
\begin{align*}
	\overset{n}{\underset{i, j=1}{\sum}} \left \la k(z_i, z_j)u_j, u_i \right \ra
	\left \la \left(\Upsilon_{12}(z_i)\Upsilon_{12}(z_j)^*-\Upsilon_{12}(z_i)\Upsilon_{31}^{(r)}(z_i) \Upsilon_{31}^{(r)}(z_j)^*\Upsilon_{12}(z_j)^*\right)v_j,v_i \right \ra \geq 0
\end{align*}
for every $r \in (0, 1)$. Letting $r \to 1$, we have that
\begin{align*}
	\overset{n}{\underset{i, j=1}{\sum}} \left \la k(z_i, z_j)u_j, u_i \right \ra
	\left \la \left(\Upsilon_{12}(z_i)\Upsilon_{12}(z_j)^*-\Upsilon_{12}(z_i)\Upsilon_{31}(z_i) \Upsilon_{31}(z_j)^*\Upsilon_{12}(z_j)^*\right)v_j,v_i \right \ra \geq 0
\end{align*}
and so, $\begin{bmatrix}
	\left(\Upsilon_{12}(z_i)\Upsilon_{12}(z_j)^*-\Upsilon_{32}(z_i)\Upsilon_{32}(z_j)^*\right)\otimes k(z_i, z_j)
\end{bmatrix} \geq 0$.

\medskip

\noindent	$(2) \implies (3)$. Consider the map 
\[
g: \Pe \times \Pe \to \mathcal{B}(\LS_2) \quad \text{defined as} \quad g(z, w)=\Upsilon_{12}(z)\Upsilon_{12}(w)^*-\Upsilon_{32}(z)\Upsilon_{32}(w)^*.
\]
Evidently, $g$ is a self-adjoint function such that $g \oslash k$ is positive semi-definite for any $\mathcal{B}(\LS_2)$ admissible kernel $k$. The desired conclusion now follows from Theorem \ref{thm_sa_P}.

\medskip

\noindent	$(3) \implies (1)$. Suppose there exist completely positive kernels $\xi, \nabla: \Pe \times \Pe \to \mathcal{B}(C(\DC), \mathcal{B}(\LS_2))$ and a weak kernel $\delta: \Pe \times \Pe \to \mathcal{B}(\LS_2)$ such that $\Upsilon_{12}(z)\Upsilon_{12}(w)^*-\Upsilon_{32}(z)\Upsilon_{32}(w)^*
$ equals
\begin{align}\label{eqn_2004}
\xi(z, w)(1-\mathrm{J}(z)\overline{\mathrm{J}(w)})+\nabla(z, w)(1-\mathrm{j}(z)\overline{\mathrm{j}(w)})  +\left(1-\frac{z^{(2)}\overline{w}^{(2)}}{4}\right)\delta(z, w).
\end{align}
By Proposition \ref{prop_prelim_I_P}, there exist Hilbert space $\KS_1, \KS_2$ and maps $L_i: \Pe \to \mathcal{B}(C(\DC), \mathcal{B}(\KS_i, \LS_2))$ for $i=1, 2$ such that for every $z, w \in \Pe$ and $f, h \in C(\DC)$,
\begin{align}\label{eqn_2005}
	\xi(z, w)(f\overline{h})=L_1(z)(f)(L_1(w)(h))^* \quad  \text{and} \quad \nabla(z, w)(f\overline{h})=L_2(z)(f)(L_2(w)(h))^*.
	\end{align} 
Also, there exist unital $*$-representations $\rho_i: C(\DC)\to \mathcal{B}(\KS_i)$ for $i=1, 2$ such that for every $z \in \Pe$ and $f, h \in C(\DC)$, we have 
\begin{align}
(L_1(z)(fh))^*=\rho_1(f)^*(L_1(z)(h))^* \quad \text{and} \quad (L_2(z)(fh))^*=\rho_2(f)^*(L_2(z)(h))^*.
\end{align}
We have by Theorem 2.62 in \cite{Agler_McCarthy} that there exist a Hilbert space $\KS_3$ and a map $G: \Pe \to \mathcal{B}(\KS_3, \LS_2)$ such that $\delta(z, w)=G(z)G(w)^*$. Set $\KS=\KS_1 \oplus \KS_2 \oplus \KS_3$. Using \eqref{eqn_2005}, we can re-write \eqref{eqn_2004} as 
\begin{align*}
&\Upsilon_{12}(z)\Upsilon_{12}(w)^*+L_1(z)(\mathrm{J}(z))(L_1(w)(\mathrm{J}(w)))^*+L_2(z)(\mathrm{j}(z))(L_2(w)(\mathrm{j}(w)))^*+\frac{z^{(2)}\overline{w}^{(2)}}{4}G(z)G(w)^*
\\
&=
\Upsilon_{32}(z)\Upsilon_{32}(w)^*+L_1(z)(1)(L_1(w)(1))^*+L_2(z)(1)(L_2(w)(1))^*+G(z)G(w)^*
\end{align*}
for every $z, w \in \Pe$. Consequently, we have for every $z, w \in \Pe$ and $v \in \LS_2$ that
{\small
	\setlength{\arraycolsep}{3pt}
\begin{align}\label{eqn_RP_004}
	& \left\langle \begin{bmatrix}
		\Upsilon_{12}(z)^*v \\
		Y(z)^*
		\begin{bmatrix}
			(L_1(z)(1))^*v \\
			(L_2(z)(1))^*v \\
			G(z)^*v
		\end{bmatrix}
	\end{bmatrix}, \begin{bmatrix}
	\Upsilon_{12}(w)^*v \\
	Y(w)^*
	\begin{bmatrix}
		(L_1(w)(1))^*v \\
	(L_2(w)(1))^*v \\
	G(w)^*v
	\end{bmatrix}
	\end{bmatrix} \right\rangle =\left\langle \begin{bmatrix}
	\Upsilon_{32}(z)^*v \\
		(L_1(z)(1))^*v \\
		(L_2(z)(1))^*v \\
		G(z)^*v
	\end{bmatrix}, \begin{bmatrix}
	\Upsilon_{32}(w)^*v \\
		(L_1(w)(1))^*v \\
		(L_2(w)(1))^*v \\
		G(w)^*v
	\end{bmatrix} \right\rangle, 
\end{align}}
where $Y(z)=\begin{bmatrix} \rho_1(\mathrm{J}(z)) & 0 & 0\\ 
	0 & \rho_2(\mathrm{j}(z)) & 0\\
	0 & 0 & (z^{(2)}\slash 2)I_{\KS_3}
\end{bmatrix}$. Consider the spaces given by
{\small 
	\setlength{\arraycolsep}{3pt}
	\begin{align*}
	\KS_d
	= \overline{\text{span}} \left\{
	\begin{bmatrix}
		\Upsilon_{12}(z)^*v \\
		Y(z)^*
		\begin{bmatrix}
			(L_1(z)(1))^*v \\
			(L_2(z)(1))^*v \\
			G(z)^*v
		\end{bmatrix}
	\end{bmatrix}
	: z \in \Pe, v \in \LS_2
	\right\} \ \text{and} \ \KS_r
	= \overline{\text{span}} \left\{
	\begin{bmatrix}
		\Upsilon_{32}(z)^*v \\
		(L_1(z)(1))^*v \\
		(L_2(z)(1))^*v \\
		G(z)^*v
	\end{bmatrix}
	: z \in \Pe, v \in \LS_2
	\right\},
\end{align*}
}
which are subspaces of $\LS_1 \oplus \KS$ and $\LS_3 \oplus \KS$, respectively.
It follows from \eqref{eqn_RP_004} that the linear operator $V: \KS_d \to \KS_r$ defined by
\[
V\begin{bmatrix}
	\Upsilon_{12}(z)^*v \\
	Y(z)^*
	\begin{bmatrix}
		(L_1(z)(1))^*v \\
		(L_2(z)(1))^*v \\
		G(z)^*v
	\end{bmatrix}
\end{bmatrix}
=\begin{bmatrix}
	\Upsilon_{32}(z)^*v \\
	(L_1(z)(1))^*v \\
	(L_2(z)(1))^*v \\
	G(z)^*v
\end{bmatrix}
\]
is an isometry. Adding an infinite-dimensional summand to $\KS_d$, if necessary, $V$ extends to a unitary from $\LS_1 \oplus \KS$ onto $\LS_3 \oplus \KS$. In $2 \times 2$ block form, we can write $V=\begin{bmatrix} A_1 & B_1 \\ C_1 & D_1 \end{bmatrix}$. Let $v \in \LS_2$. By definition of $V$, it then follows that 
{\small 
\begin{align}\label{eqn_RP_005}
	 A_1\Upsilon_{12}(z)^*v+B_1Y(z)^*\begin{bmatrix}
	 	(L_1(z)(1))^*v \\
	 	(L_2(z)(1))^*v \\
	 	G(z)^*v
	 \end{bmatrix}&=\Upsilon_{32}(z)^*v \quad \text{and} 
	\notag \\ 
	 C_1\Upsilon_{12}(z)^*v+D_1Y(z)^*\begin{bmatrix}
	(L_1(z)(1))^*v \\
	(L_2(z)(1))^*v \\
	G(z)^*v
	\end{bmatrix}&=\begin{bmatrix}
	(L_1(z)(1))^*v \\
	(L_2(z)(1))^*v \\
	G(z)^*v
	\end{bmatrix}.
\end{align}
}
Since $\|Y(z)\|<1$ and $\|D\| \leq 1$, we have by \eqref{eqn_RP_005} that $\begin{bmatrix}
	(L_1(z)(1))^*v \\
	(L_2(z)(1))^*v \\
	G(z)^*v
\end{bmatrix}=(I-D_1Y(z)^*)^{-1}C_1\Upsilon_{12}(z)^*v$. Thus, $\Upsilon_{32}(z)^*v=A_1\Upsilon_{12}(z)^*v+B_1Y(z)^*(I-D_1Y(z)^*)^{-1}C_1\Upsilon_{12}(z)^*v$ for $z \in \Pe$ and $v \in \LS_2$. Define   
\[
\Upsilon_{31}: \Pe \to \mathcal{B}(\LS_3, \LS_1) \quad \text{as} \quad \Upsilon_{31}(z)^*=A_1+B_1Y(z)^*(I-D_1Y(z)^*)^{-1}C_1.
\]
It is not difficult to see that $\Upsilon_{12}(z)\Upsilon_{31}(z)=\Upsilon_{32}(z)$ for every $z \in \Pe$, and that 
\[
\Upsilon_{31}(z)=A_1^*+C_1^*Y(z)(I-D_1^*Y(z))^{-1}B_1^*=A+BY(z)(I-DY(z))^{-1}C,
\]
where $U=\begin{bmatrix}
	A & B \\ 
	C & D
\end{bmatrix}=\begin{bmatrix}
A_1^* & C_1^*\\
B_1^* & D_1^*
\end{bmatrix}=V^*$ is a unitary from $\LS_3 \oplus \KS$ onto $\LS_1 \oplus \KS$. By Theorem \ref{thm_realization_P}, $\Upsilon_{31} \in SA_{\Pe}(\LS_3, \LS_1)$, which completes the proof. 
\end{proof}
	
Now we are in a position of presenting the desired Toeplitz corona theorem on $\Pe$, which is the main result of this paper.
	
\begin{thm}\label{thm_TC_P_II}
	Let $\varphi_1, \dotsc, \varphi_n \in H^\infty(\Pe)$ and let $\epsilon>0$. Then the following statements are equivalent:
	\begin{enumerate}[leftmargin=*]
		\item[$(1)$] there exist $f_1, \dotsc, f_n \in H^\infty(\Pe)$ such that  
		\[
		\begin{bmatrix}
			 f_1 & \dotsc & f_n
		\end{bmatrix}^t \in \frac{1}{\epsilon}SA_\Pe(\C, \C^n), \ \text{i.e.,} \ \left\|\begin{bmatrix}
		f_1(\underline{T}) \\
		\vdots \\
		f_n(\underline{T})
		\end{bmatrix} \right\| \leq \frac{1}{\epsilon} \ \text{for every} \ \underline{T} \in \mathfrak{M}_\Pe
		\] 
		and $\varphi_1(z)f_1(z)+ \dotsc + \varphi_n(z)f_n(z)=1$ for every $z \in \Pe$;
		
		\item[$(2)$] the map 
		$
		(z, w) \mapsto \left(\overset{n}{\underset{j=1}{\sum}} \varphi_j(z)\overline{\varphi_j(w)}-\epsilon^2\right)k(z, w)
		$ on $\Pe \times \Pe$	is positive semi-definite for every $k \in AK(\Pe)$;
		
		\item[$(3)$] there exist positive kernels $\xi, \nabla: \Pe \times \Pe \to C(\DC)^*$ and a weak kernel $\delta: \Pe \times \Pe \to \C$ such that
	{\small	\begin{align*}
			\overset{n}{\underset{j=1}{\sum}} \varphi_j(z)\overline{\varphi_j(w)}-\epsilon^2
			=\xi(z, w)(1-\mathrm{J}(z)\overline{\mathrm{J}(w)})+\nabla(z, w)(1-\mathrm{j}(z)\overline{\mathrm{j}(w)})+\left(1-\frac{z^{(2)}\overline{w}^{(2)}}{4}\right)\delta(z, w)
		\end{align*}
	}
	for all $z, w \in \Pe$, where the maps $z\mapsto \mathrm{J}(z)$ and $z\mapsto \mathrm{j}(z)$ are as in \eqref{eqn_J(z)}.	 	
	\end{enumerate}
\end{thm}	

\begin{proof}
	The proof follows from Theorem \ref{thm_TC_P} by choosing the spaces $\LS_1=\C^n, \LS_2=\LS_3=\C$, and defining $\Upsilon_{32}(z)=\epsilon, \Upsilon_{12}(z)=\begin{bmatrix}
		\varphi_1(z) & \dotsc & \varphi_n(z)
	\end{bmatrix}$ and $\Upsilon_{31}(z)=\epsilon \begin{bmatrix}
	 f_1(z) & \dotsc & f_n(z)
	\end{bmatrix}^t$ for $z \in \Pe$.
\end{proof}

The authors of \cite{Pal2026} established realization, interpolation, and extension theorems on $\Pe$, and consequently  recovered the corresponding results on the bidisc $\D^2$ and the symmetrized bidisc $\G_2$. In a similar spirit, as applications of Theorems \ref{thm_TC_P} and \ref{thm_TC_P_II}, we obtain alternative characterizations for the Toeplitz corona theorems on $\D^2$ and $\G_2$. This illustrates that the function theories on $\D^2$ and $\G_2$ can be studied within a unified framework through the function theory on $\Pe$.

\subsection{The bidisc case.} The Toeplitz corona theorem on the bidisc \cite{Agler_McCarthyI} asks the following: given $\varphi_1, \dotsc, \varphi_n$ in $H^\infty(\D^2)$ and $\epsilon>0$, when does there exist functions $f_1, \dotsc, f_n$ in $H^\infty(\D^2)$ such that  
\[
\sup_{z\in \D^2} \left[|f_1(z)|^2+\dotsc+|f_n(z)|^2 \right]\leq \frac{1}{\epsilon} \  \ \text{and} \ \ \varphi_1f_1+ \dotsc + \varphi_nf_n=1?
\] 
The authors of \cite{Agler_McCarthyI} provided necessary and sufficient conditions for the existence of such $f_1, \dotsc, f_n$ in $H^\infty(\D^2)$. We present alternative characterizations for the Toeplitz corona theorem on $\D^2$ in terms of admissible kernels on $\Pe$. We have by Theorem \ref{thm_connect_P} that $(z^{(1)}, 0, z^{(3)}) \in \Pe$ if and only if $(z^{(1)}, z^{(3)}) \in \D^2$. Moreover, if $(z^{(1)}, z^{(2)}, z^{(3)})\in \Pe$, then $(z^{(1)}, z^{(3)}) \in \D^2$. Thus, a holomorphic map $g: \D^2 \to \C$ induces a holomorphic map $g\circ \theta_{\Pe \to \D^2}$ on $\Pe$, where 
\[
\theta_{\Pe \to \D^2}: \Pe \to \D^2, (z^{(1)}, z^{(2)}, z^{(3)}) \mapsto (z^{(1)}, z^{(3)}).
\]
This gives a way through which we transfer the Toeplitz corona theorem on $\D^2$ to the corresponding result on $\Pe$. Before going further, we briefly discuss Ando's inequality \cite{Ando} for operator-valued holomorphic functions on $\D^2$. For Hilbert spaces $\LS, \LS'$, we denote by $H^\infty(\D^2, \mathcal{B}(\LS, \LS'))$ the space of all bounded $\mathcal{B}(\LS, \LS')$-valued holomorphic functions on $\D^2$. The success of unitary dilation for a commuting pair of contractions $(T_1, T_2)$ yields that $\DC^2$ is a complete spectral set for $(T_1, T_2)$, i.e,
\begin{align}\label{eqn_Ando}
	\left \|[f_{ij}(T_1, T_2)]_{i, j=1}^n\right\| \leq \sup\left\{\left\|[f_{ij}(z)]_{i, j=1}^n\right\| : z \in \D^2  \right\}
\end{align}
for all matricial functions $[f_{ij}]_{n \times n}$ with each $f_{ij}$ in the rational algebra $\text{Rat}(\DC^2)$. Ando's inequality as in \eqref{eqn_Ando} can be extended to operator-valued holomorphic maps on $\D^2$. We briefly explain it here. Let $f \in H^\infty(\D^2, \mathcal{B}(\LS_1, \LS_2))$ and let $(T_1, T_2)$ be a commuting pair of contractions. For $0< r<1$, the map $f_r: \DC^2 \to \mathcal{B}(\LS_1, \LS_2)$ defined as $f_r(z^{(1)}, z^{(2)})=f(rz^{(1)}, rz^{(2)})$ is holomorphic. Consider orthonormal sets $\{e_1, \dotsc, e_m\}$ and $\{e_1', \dotsc, e_n'\}$ in $\LS_1$ and $\LS_2$, respectively. Define $F_r: \DC^2 \to M_{n \times m}(\C)$ as $\displaystyle F_r(z)=[\la f_r(z)e_i, e_j'\ra ]_{j, i}$. Let $z \in \DC^2$. Then
{\small 
	\begin{align*}
		\left|\la F_r(z)(h_1,\dotsc,h_m),(h'_1,\dotsc,h'_n)\ra\right| 
		=\left|	\sum_{i=1}^m\sum_{j=1}^n
		\la f_r(z)(h_i\otimes e_i),h'_j\otimes e'_j\ra \right| 
		&=
		\left|\left\la
		f_r(z)
		\left(
		\sum_{i=1}^m h_i\otimes e_i
		\right),
		\sum_{j=1}^n h'_j\otimes e'_j
		\right\ra\right|\\
	& \leq \|f_r\|_{\infty, \DC^2}
	\left\|
	\sum_{i=1}^m h_i\otimes e_i
	\right\| 
	\left\|
	\sum_{j=1}^n h'_j\otimes e'_j
	\right\|\\
	&= \|f_r\|_{\infty, \DC^2}
	\|(h_1,\dotsc,h_m)\|
	\,
	\|(h'_1,\dotsc,h'_n)\|
\end{align*}
}
and so, $\|F_r\|_{\infty, \DC^2} \leq \|f_r\|_{\infty, \DC^2}$. Since $\DC^2$ is a complete spectral set for $(T_1, T_2), \|F_r(T_1, T_2)\| \leq \|F_r\|_{\infty, \DC^2} \leq \|f_r\|_{\infty, \DC^2}$. Following the arguments as in the proof of Lemma \ref{lem_201}, we have that $\|f_r(T_1, T_2)\| \leq \|f_r\|_{\infty, \DC^2} \leq \|f\|_{\infty, \D^2}$. Taking $r\to 1$, $\|f(T_1, T_2)\| \leq \|f\|_{\infty, \D^2}$ and so, \eqref{eqn_Ando} holds for bounded operator-valued holomorphic maps on $\D^2$. Next, we have the following result.

\begin{thm}\label{thm_TC_D2}
	Let $\LS_1, \LS_2$ and $\LS_3$ be Hilbert spaces. For functions $\Upsilon_{12}: \D^2 \to \mathcal{B}(\LS_1, \LS_2)$ and $\Upsilon_{32}: \D^2 \to \mathcal{B}(\LS_3, \LS_2)$, the following are equivalent: 
	\begin{enumerate}[leftmargin=*]
		\item[$(1)$] there exists $\Upsilon_{31} \in H^\infty(\D^2, \mathcal{B}(\LS_3, \LS_1))$ with $\|\Upsilon_{31}\|_{\infty, \D^2} \leq 1$ such that $\Upsilon_{12}(z)\Upsilon_{31}(z)=\Upsilon_{32}(z)$ for every $z \in \D^2$;
		
		\item[$(2)$] for every subset $F=\{z_1, \dotsc, z_n\} \subset \Pe$ and 		for every $\mathcal{B}(\LS_2)$-valued admissible kernel $k$ on $\Pe$,
		\[
		\begin{bmatrix}
			\left(\Upsilon_{12}\left(z_i^{(1)}, z_i^{(3)}\right)\Upsilon_{12}\left(z_j^{(1)}, z_j^{(3)}\right)^*-\Upsilon_{32}\left(z_i^{(1)}, z_i^{(3)}\right)\Upsilon_{32}\left(z_j^{(1)}, z_j^{(3)}\right)^*\right)\otimes k(z_i, z_j)
		\end{bmatrix}_{i, j=1}^n \geq 0;
		\]
		
		\item[$(3)$] there exist completely positive kernels $\xi, \nabla: \Pe \times \Pe \to \mathcal{B}(C(\DC), \mathcal{B}(\LS_2))$ and a weak kernel $\delta: \Pe \times \Pe \to \mathcal{B}(\LS_2)$ such that for all $z, w \in \Pe$,
		\begin{align*}
			&\Upsilon_{12}\left(z^{(1)}, z^{(3)}\right)\Upsilon_{12}\left(w^{(1)}, w^{(3)}\right)^*-\Upsilon_{32}\left(z^{(1)}, z^{(3)}\right)\Upsilon_{32}\left(w^{(1)}, w^{(3)}\right)^*\\
			&=\xi(z, w)(1-\mathrm{J}(z)\overline{\mathrm{J}(w)})+\nabla(z, w)(1-\mathrm{j}(z)\overline{\mathrm{j}(w)})+\left(1-\frac{z^{(2)}\overline{w}^{(2)}}{4}\right)\delta(z, w),
		\end{align*}
		where the maps $z\mapsto \mathrm{J}(z)$ and $z\mapsto \mathrm{j}(z)$ are as in \eqref{eqn_J(z)}.	 			
	\end{enumerate}
\end{thm}	

\begin{proof}
	We prove $(1) \implies (2) \implies (3) \implies (1)$. Given a map $\Upsilon:\D^2 \to \mathcal{B}(\LS, \LS')$, we denote by $\widehat{\Upsilon}: \Pe \to \mathcal{B}(\LS, \LS')$ the map $\widehat{\Upsilon}(z)=\Upsilon\circ \theta_{\Pe \to \D^2}(z)=\Upsilon(z^{(1)}, z^{(3)})$ for $z=(z^{(1)}, z^{(2)}, z^{(3)}) \in \Pe$.
	
	\medskip 
	
	\noindent $(1) \implies (2)$. Suppose there exists $\Upsilon_{31} \in H^\infty(\D^2, \mathcal{B}(\LS_3, \LS_1))$ with $\|\Upsilon_{31}\|_{\infty, \D^2} \leq 1$ such that $\Upsilon_{12}(z)\Upsilon_{31}(z)=\Upsilon_{32}(z)$ for every $z \in \D^2$. Evidently, $\widehat{\Upsilon}_{31} \in H^\infty(\Pe, \mathcal{B}(\LS_3, \LS_1))$ such that $\widehat{\Upsilon}_{32}=\widehat{\Upsilon}_{12}\widehat{\Upsilon}_{31}$. It follows from Example 3.7 in \cite{Pal2026} that $(T_1, T_3)$ is a commuting pair of contractions for every $\underline{T}=(T_1, T_2, T_3) \in \mathfrak{M}_{\Pe}$. Let $\underline{T}=(T_1, T_2, T_3)\in \mathfrak{M}_{\Pe}$ be acting on a Hilbert space $\HS$. By above discussion, $\|\widehat{\Upsilon}_{31}(T_1, T_2, T_3)\|=\|\Upsilon_{31}(T_1, T_3)\| \leq \|\Upsilon_{31}\|_{\infty, \D^2} \leq 1$ and so, $\widehat{\Upsilon}_{31} \in SA_\Pe(\LS_3, \LS_1)$. Applying the part $(1) \implies (2)$ of Theorem \ref{thm_TC_P} to $(\widehat{\Upsilon}_{12}, \widehat{\Upsilon}_{32}, \widehat{\Upsilon}_{31})$ gives the desired conclusion. 
	
	\medskip 
	
	\noindent $(2) \implies (3)$. It follows by applying part $(2) \implies (3)$ of Theorem \ref{thm_TC_P} to $(\widehat{\Upsilon}_{12}, \widehat{\Upsilon}_{32}, \widehat{\Upsilon}_{31})$.
	
	\medskip 
	
	\noindent $(3) \implies (2)$. Suppose condition $(3)$ as in the statement of the theorem holds. Then there exist completely positive kernels $\xi, \nabla: \Pe \times \Pe \to \mathcal{B}(C(\DC), \mathcal{B}(\LS_2))$ and a $\mathcal{B}(\LS_2)$-valued weak kernel $\delta: \Pe \times \Pe \to \mathcal{B}(\LS_2)$ such that for all $z, w \in \Pe$,
	{\small 
		\begin{align*}
			\widehat{\Upsilon}_{12}(z)\widehat{\Upsilon}_{12}(w)^*-\widehat{\Upsilon}_{32}(z)\widehat{\Upsilon}_{32}(w)^*=\xi(z, w)(1-\mathrm{J}(z)\overline{\mathrm{J}(w)})+\nabla(z, w)(1-\mathrm{j}(z)\overline{\mathrm{j}(w)})+\left(1-\frac{z^{(2)}\overline{w}^{(2)}}{4}\right)\delta(z, w).
		\end{align*}
	}
	By the part $(3) \implies (1)$ of Theorem \ref{thm_TC_P}, there exists $\Upsilon_0 \in SA_{\Pe}(\LS_3, \LS_1)$ such that $\widehat{\Upsilon}_{12}(z)\Upsilon_{0}(z)=\widehat{\Upsilon}_{32}(z)$ for every $z \in \Pe$. Define $\Upsilon_{31}: \D^2 \to \mathcal{B}(\LS_3, \LS_1)$ as $\Upsilon_{31}(z^{(1)}, z^{(3)})=\Upsilon_0(z^{(1)}, 0, z^{(3)})$. It is not difficult to see that $\Upsilon_{31} \in H^\infty(\D^2, \mathcal{B}(\LS_3, \LS_1))$ with $\|\Upsilon_{31}\|_{\infty, \D^2} \leq 1$. The proof is now complete.
\end{proof}	

We now recover Toeplitz corona theorem on $\D^2$ from the pentablock framework.

\begin{thm}
	Let $\varphi_1, \dotsc, \varphi_n \in H^\infty(\D^2)$ and let $\epsilon>0$. Then the following are equivalent:
	\begin{enumerate}[leftmargin=*]
		\item[$(1)$] there exist $f_1, \dotsc, f_n \in H^\infty(\D^2)$ such that $
		\displaystyle \sup_{z\in \D^2} \left[|f_1(z)|^2+\dotsc+|f_n(z)|^2 \right]\leq \frac{1}{\epsilon}$ and $\overset{n}{\underset{i=1}{\sum}}\varphi_if_i=1$;
		
		\item[$(2)$] the map 
		$
		(z, w) \mapsto \left(\overset{n}{\underset{j=1}{\sum}} \varphi_j(z^{(1)}, z^{(3)})\overline{\varphi_j(w^{(1)}, w^{(3)})}-\epsilon^2\right)k(z, w)$ on $\Pe \times \Pe$ is positive semi-definite for every $k \in AK(\Pe)$;
		
		\item[$(3)$] there exist positive kernels $\xi, \nabla: \Pe \times \Pe \to C(\DC)^*$ and a weak kernel $\delta: \Pe \times \Pe \to \C$ such that
	{\small	\begin{align*}
			\overset{n}{\underset{j=1}{\sum}} \varphi_j(z^{(1)}, z^{(3)})\overline{\varphi_j(w^{(1)}, w^{(3)})}-\epsilon^2
			 =\xi(z, w)(1-\mathrm{J}(z)\overline{\mathrm{J}(w)})+\nabla(z, w)(1-\mathrm{j}(z)\overline{\mathrm{j}(w)})+\left(1-\frac{z^{(2)}\overline{w}^{(2)}}{4}\right)\delta(z, w)
		\end{align*}}
		for all $z, w \in \Pe$, where the maps $z\mapsto \mathrm{J}(z)$ and $z\mapsto \mathrm{j}(z)$ are as in \eqref{eqn_J(z)}.	 
	\end{enumerate}
\end{thm}		
	
\begin{proof}
The desired conclusion follows by applying Theorem \ref{thm_TC_D2} with $\LS_1=\C^n, \LS_2=\LS_3=\C$, and taking $\Upsilon_{32}(z)=\epsilon, \Upsilon_{12}(z)=\begin{bmatrix}
	\varphi_1(z) & \dotsc & \varphi_n(z)
\end{bmatrix}$ and $\Upsilon_{31}(z)=\epsilon \begin{bmatrix}
	f_1(z) & \dotsc & f_n(z)
\end{bmatrix}^t$ for $z \in \D^2$. 
\end{proof}	
	
\subsection{The symmetrized bidisc case.} 	Let $\LS_1, \LS_2$ and $\LS_3$ be complex separable Hilbert spaces. For functions $\Upsilon_{12}: \G_2 \to \mathcal{B}(\LS_1, \LS_2)$ and $\Upsilon_{32}: \G_2 \to \mathcal{B}(\LS_3, \LS_2)$, one asks if there exists
\begin{align}\label{eqn_TCG_2}
\text{$\Upsilon_{31} \in H^\infty(\G_2, \mathcal{B}(\LS_3, \LS_1))$ with $\|\Upsilon_{31}\|_{\infty, \G_2} \leq 1$ such that $\Upsilon_{12}(z)\Upsilon_{31}(z)=\Upsilon_{32}(z)$} 
\end{align} 
for every $z \in \G_2$. Here, $H^\infty(\G_2, \mathcal{B}(\LS, \LS'))$ is the space of all bounded $\mathcal{B}(\LS, \LS')$-valued holomorphic functions on $\G_2$. The authors of \cite{Tirtha_Sau} provided necessary and sufficient conditions for the existence of $\Upsilon_{31}$ as in \eqref{eqn_TCG_2}. As a consequence, they obtained Toeplitz corona theorem on $\G_2$. We present alternative characterizations to this problem through the pentablock framework. We recall from Theorem \ref{thm_connect_P} that $(z^{1)}, z^{(2)}) \in \G_2$ if and only if $(0, z^{(1)}, z^{(2)}) \in \Pe$. A holomorphic map $g: \G_2 \to \C$ induces a holomorphic map $g\circ \theta_{\Pe \to \G_2}$ on $\Pe$, where 
\[
\theta_{\Pe \to \G_2}: \Pe \to \G_2, (z^{(1)}, z^{(2)}, z^{(3)}) \mapsto (z^{(2)}, z^{(3)}).
\]
We briefly discuss operator theory on the symmetrized bidisc from \cite{AglerYoung, AglerYoung2003}. A commuting pair of operators $(S, P)$ is said to be a \textit{$\Gamma$-contraction} if $\|p(S, P)\| \leq \|p\|_{\infty, \Gamma}$ for all holomorphic polynomials $p$ in two variables. Here, $\Gamma$ denotes the closed symmetrized bidisc $\overline{\G}_2$. It was proved in \cite{AglerYoung} that $\Gamma$ is a complete spectral set for a $\Gamma$-contraction $(S, P)$, i.e., 
\begin{align}\label{eqn_309}
\left \|[f_{ij}(S, P)]_{i, j=1}^n\right\| \leq \sup\left\{\left\|[f_{ij}(z)]_{i, j=1}^n\right\| : z \in \G_2  \right\}
\end{align}
for all matricial functions $[f_{ij}]_{n \times n}$ with each $f_{ij}$ in the rational algebra $\text{Rat}(\Gamma)$ over $\Gamma$. It is not difficult to see that \eqref{eqn_309} also holds for operator-valued holomorphic maps on $\G_2$. We briefly sketch an argument here. Let $f \in H^\infty(\G_2, \mathcal{B}(\LS_1, \LS_2))$ and let $(S, P)$ be a $\Gamma$-contraction. For $0<r<1$ and $(z^{(1)}, z^{(2)}) \in \Gamma$, we have that $(rz^{(1)}, r^2z^{(2)})\in \G_2$ and so, the map $f_r: \Gamma \to \mathcal{B}(\LS_1, \LS_2)$ defined by $f_r(z^{(1)}, z^{(2)})=f(rz^{(1)}, r^2z^{(2)})$ is holomorphic. Let $\{e_1, \dotsc, e_m\} \subseteq \LS_1$ and $\{e_1', \dotsc, e_n'\} \subseteq \LS_2$ be orthonormal sets. Define $F_r: \Gamma \to M_{n \times m}(\C)$ as $\displaystyle F_r(z)=[\la f_r(z)e_i, e_j'\ra ]_{j, i}$. Let $z \in \Gamma$. Then 
{\small 
	\begin{align*}
	\la F_r(z)(h_1,\dotsc,h_m),(h'_1,\dotsc,h'_n)\ra 
	=	\sum_{i=1}^m\sum_{j=1}^n
	\la f_r(z)(h_i\otimes e_i),h'_j\otimes e'_j\ra 
	=
	\left\la
	f_r(z)
	\left(
	\sum_{i=1}^m h_i\otimes e_i
	\right),
	\sum_{j=1}^n h'_j\otimes e'_j
	\right\ra.
\end{align*}
}
Therefore,
$
	\left|
	\la F_r(z)(h_1,\dotsc,h_m),(h'_1,\dotsc,h'_n)\ra
	\right| 
	\leq  \|f_r\|_{\infty, \Gamma}
	\|(h_1,\dotsc,h_m)\|
	\,
	\|(h'_1,\dotsc,h'_n)\|
$
and $\|F_r\|_{\infty, \Gamma} \leq \|f_r\|_{\infty, \Gamma}$. Using the fact that $\Gamma$ is a complete spectral set for $(S, P)$, it follows that $\|F_r(S, P)\| \leq \|F_r\|_{\infty, \Gamma} \leq \|f_r\|_{\infty, \Gamma}$. Since the orthonormal sets $\{e_1, \dotsc, e_m\}$ and $\{e_1', \dotsc, e'_n\}$ are arbitrary, one can follow the arguments as in the converse part of Lemma \ref{lem_201} to show that $\|f_r(S, P)\| \leq \|f_r\|_{\infty, \Gamma} \leq \|f\|_{\infty, \G_2}$. Letting $r\to 1$, $\|f(S, P)\| \leq \|f\|_{\infty, \G_2}$. We are now in a position to provide an alternative proof of the following generalized version of Toeplitz corona theorem on $\G_2$ from \cite{Tirtha_Sau}. Our characterizations are written in terms of admissible kernels and completely positive kernels on $\Pe$.

\begin{thm}\label{thm_TC_G_2}
	Let $\LS_1, \LS_2$ and $\LS_3$ be Hilbert spaces. For functions $\Upsilon_{12}: \G_2 \to \mathcal{B}(\LS_1, \LS_2)$ and $\Upsilon_{32}: \G_2 \to \mathcal{B}(\LS_3, \LS_2)$, the following are equivalent: 
	\begin{enumerate}[leftmargin=*]
		\item[$(1)$] there exists $\Upsilon_{31} \in H^\infty(\G_2, \mathcal{B}(\LS_3, \LS_1))$ with $\|\Upsilon_{31}\|_{\infty, \G_2} \leq 1$ such that $\Upsilon_{12}(z)\Upsilon_{31}(z)=\Upsilon_{32}(z)$ for every $z \in \G_2$;

		\item[$(2)$] for every subset $F=\{z_1, \dotsc, z_n\} \subset \Pe$ and 		for every $\mathcal{B}(\LS_2)$-valued admissible kernel $k$ on $\Pe$,
				\[
		\begin{bmatrix}
			\left(\Upsilon_{12}\left(z_i^{(2)}, z_i^{(3)}\right)\Upsilon_{12}\left(z_j^{(2)}, z_j^{(3)}\right)^*-\Upsilon_{32}\left(z_i^{(2)}, z_i^{(3)}\right)\Upsilon_{32}\left(z_j^{(2)}, z_j^{(3)}\right)^*\right)\otimes k(z_i, z_j)
		\end{bmatrix}_{i, j=1}^n \geq 0;
		\]
		
		\item[$(3)$] there exist completely positive kernels $\xi, \nabla: \Pe \times \Pe \to \mathcal{B}(C(\DC), \mathcal{B}(\LS_2))$ and a weak kernel $\delta: \Pe \times \Pe \to \mathcal{B}(\LS_2)$ such that
		\begin{align*}
			&\Upsilon_{12}\left(z^{(2)}, z^{(3)}\right)\Upsilon_{12}\left(w^{(2)}, w^{(3)}\right)^*-\Upsilon_{32}\left(z^{(2)}, z^{(3)}\right)\Upsilon_{32}\left(w^{(2)}, w^{(3)}\right)^*\\
			&=\xi(z, w)(1-\mathrm{J}(z)\overline{\mathrm{J}(w)})+\nabla(z, w)(1-\mathrm{j}(z)\overline{\mathrm{j}(w)})+\left(1-\frac{z^{(2)}\overline{w}^{(2)}}{4}\right)\delta(z, w)
		\end{align*}
for all $z, w \in \Pe$, where the maps $z\mapsto \mathrm{J}(z)$ and $z\mapsto \mathrm{j}(z)$ are as in \eqref{eqn_J(z)}.	 		
	\end{enumerate}
\end{thm}

\begin{proof}
	We prove $(1) \implies (2) \implies (3) \implies (1)$. Given a map $\Upsilon:\G_2 \to \mathcal{B}(\LS, \LS')$, we denote by $\widetilde{\Upsilon}: \Pe \to \mathcal{B}(\LS, \LS')$ the map $\widetilde{\Upsilon}(z)=\Upsilon\circ \theta_{\Pe \to \G_2}(z)$ for $z \in \G_2$.

\medskip 

\noindent $(1) \implies (2)$. Suppose there exists $\Upsilon_{31} \in H^\infty(\G_2, \mathcal{B}(\LS_3, \LS_1))$ with $\|\Upsilon_{31}\|_{\infty, \G_2} \leq 1$ such that $\Upsilon_{12}(z)\Upsilon_{31}(z)=\Upsilon_{32}(z)$ for every $z \in \G_2$. It is easy to see that $\widetilde{\Upsilon}_{31} \in H^\infty(\Pe, \mathcal{B}(\LS_3, \LS_1))$ and $\widetilde{\Upsilon}_{32}=\widetilde{\Upsilon}_{12}\widetilde{\Upsilon}_{31}$. We now show that $\widetilde{\Upsilon}_{31} \in SA_\Pe(\LS_3, \LS_1)$. Let $\underline{T}=(T_1, T_2, T_3) \in \mathfrak{M}_{\Pe}$ be acting on a Hilbert space $\HS$. Then $\|T_2\| < 2$ and $\|\Phi_\al(T_2, T_3)\| < 1$ for all $\alpha \in \DC$. It follows from Theorem 1.5 in \cite{AglerYoung2003} that $(T_2, T_3)$ is a $\Gamma$-contraction. By above discussion, $\|\widetilde{\Upsilon}_{31}(T_1, T_2, T_3)\|=\|\Upsilon_{31}(T_2, T_3)\| \leq \|\Upsilon_{31}\|_{\infty, \G_2} \leq 1$ and so, $\widetilde{\Upsilon}_{31} \in SA_\Pe(\LS_3, \LS_1)$. Applying the part $(1) \implies (2)$ of Theorem \ref{thm_TC_P} to the triple $(\widetilde{\Upsilon}_{12}, \widetilde{\Upsilon}_{32}, \widetilde{\Upsilon}_{31})$ of vector-valued functions on $\Pe$ gives the desired conclusion. 

\medskip 

\noindent $(2) \implies (3)$. It follows by using part $(2) \implies (3)$ of Theorem \ref{thm_TC_P} with $(\widetilde{\Upsilon}_{12}, \widetilde{\Upsilon}_{32}, \widetilde{\Upsilon}_{31})$.

\medskip 

\noindent $(3) \implies (2)$. By given hypothesis, we have a weak kernel $\delta: \Pe \times \Pe \to \mathcal{B}(\LS_2)$ and completely positive kernels $\xi, \nabla: \Pe \times \Pe \to \mathcal{B}(C(\DC), \mathcal{B}(\LS_2))$ such that for all $z, w \in \Pe$,
{\small 
	\begin{align*}
	\widetilde{\Upsilon}_{12}(z)\widetilde{\Upsilon}_{12}(w)^*-\widetilde{\Upsilon}_{32}(z)\widetilde{\Upsilon}_{32}(w)^*=\xi(z, w)(1-\mathrm{J}(z)\overline{\mathrm{J}(w)})+\nabla(z, w)(1-\mathrm{j}(z)\overline{\mathrm{j}(w)})+\left(1-\frac{z^{(2)}\overline{w}^{(2)}}{4}\right)\delta(z, w).
\end{align*}
}
We have by Theorem \ref{thm_TC_P} that there exists $\Upsilon_0 \in SA_{\Pe}(\LS_3, \LS_1)$ such that $\widetilde{\Upsilon}_{12}\Upsilon_{0}=\widetilde{\Upsilon}_{32}$. Define $\Upsilon_{31}: \G_2 \to \mathcal{B}(\LS_3, \LS_1)$ as $\Upsilon_{31}(z^{(2)}, z^{(3)})=\Upsilon_0(0, z^{(2)}, z^{(3)})$ and the desired conclusion follows. 
\end{proof}	

We now present Toeplitz corona theorem on $\G_2$ in terms of kernels on $\Pe$.

\begin{thm}
	Let $\varphi_1, \dotsc, \varphi_n \in H^\infty(\G_2)$ and let $\epsilon>0$. Then the following are equivalent:
	\begin{enumerate}[leftmargin=*]
		\item[$(1)$] there exist $f_1, \dotsc, f_n \in H^\infty(\G_2)$ such that $
\displaystyle \sup_{z\in \G_2} \left[|f_1(z)|^2+\dotsc+|f_n(z)|^2 \right]\leq \frac{1}{\epsilon}$ and $\overset{n}{\underset{i=1}{\sum}}\varphi_if_i=1$;
		
		\item[$(2)$] the map 
		$
		(z, w) \mapsto \left(\overset{n}{\underset{j=1}{\sum}} \varphi_j(z^{(2)}, z^{(3)})\overline{\varphi_j(w^{(2)}, w^{(3)})}-\epsilon^2\right)k(z, w)$ on $\Pe \times \Pe
		$ 
		is positive semi-definite for every $k \in AK(\Pe)$;
		
		\item[$(3)$] there exist positive kernels $\xi, \nabla: \Pe \times \Pe \to C(\DC)^*$ and a weak kernel $\delta: \Pe \times \Pe \to \C$ such that
	{\small	\begin{align*}
			\overset{n}{\underset{j=1}{\sum}} \varphi_j(z^{(2)}, z^{(3)})\overline{\varphi_j(w^{(2)}, w^{(3)})}-\epsilon^2
			 =\xi(z, w)(1-\mathrm{J}(z)\overline{\mathrm{J}(w)})+\nabla(z, w)(1-\mathrm{j}(z)\overline{\mathrm{j}(w)})+\left(1-\frac{z^{(2)}\overline{w}^{(2)}}{4}\right)\delta(z, w)
		\end{align*}}
		for all $z, w \in \Pe$, where the maps $z\mapsto \mathrm{J}(z)$ and $z\mapsto \mathrm{j}(z)$ are as in \eqref{eqn_J(z)}.	 
	\end{enumerate}
\end{thm}	

\begin{proof}
	The desired conclusion follows by applying Theorem \ref{thm_TC_G_2} with $\LS_1=\C^n, \LS_2=\LS_3=\C$, and taking $\Upsilon_{32}(z)=\epsilon, \Upsilon_{12}(z)=\begin{bmatrix}
		\varphi_1(z) & \dotsc & \varphi_n(z)
	\end{bmatrix}$ and $\Upsilon_{31}(z)=\epsilon \begin{bmatrix}
		f_1(z) & \dotsc & f_n(z)
	\end{bmatrix}^t$ for $z \in \G_2$. 
\end{proof}

\bigskip
	
\noindent \textbf{Funding.} The first named author is supported in part by the “Core Research Grant” with Award No. CRG/2023/005223 from Anusandhan National Research Foundation (ANRF) of Govt. of India. The second named author is supported via the IIT Bombay RDF Grant of the first named author with Project Code RI/0115-10001427.

\vspace{0.2cm} 

\section{Data Availability Statement}

\noindent (1) Data sharing is not applicable to this article as no datasets were generated or analysed during the current study.\\

\noindent (2) In case any datasets are generated during and/or analysed during the current study, they must be available from the corresponding author on reasonable request.

\vspace{0.2cm}

\section{Declarations}

\vspace{0.2cm}

\noindent \textbf{Ethical Approval.} This declaration is not applicable.\\

\noindent \textbf{Competing interests.} There are no competing interests.\\

\noindent \textbf{Authors' contributions.} This declaration is not applicable.

\end{document}